\documentclass{amsart}
\usepackage{amsmath}
\usepackage{amsfonts}
\usepackage{amssymb}
\usepackage{hyperref}
\usepackage{float}
\author{Oliver Planes Osorio} 
\email{oliver.planes@upc.edu}
\date{}
\title{Expansion series of $f(x)=x^x$ and characterization of its coefficients}

\newtheorem{theorem}{Theorem}[section]
\newtheorem{lemma}[theorem]{Lemma}

\theoremstyle{definition}
\newtheorem{definition}[theorem]{Definition}

\newtheorem{proposition}[theorem]{Proposition}

\theoremstyle{remark}

\newtheorem{corollary}[theorem]{Corollary}

\numberwithin{equation}{section}

\newcommand{\oli}[2]{\Omega(#1,#2)}
\newcommand{\olipt}{\Omega(\cdot,\cdot)}

\newcommand{\del}[2]{\Delta_{#2}^{(#1)}(x)}

\newcommand{\dela}[2]{\Delta_{#2}^{(#1)}(a)}
\newcommand{\delx}[3]{\Delta_{#2}^{(#1)}(#3)}

\begin{document}
\maketitle
\begin{small}
Departament de Matem\`atica Aplicada III, Universitat Polit\`ecnica de Catalunya-Barcelona Tech, EET-Campus Terrassa, Colom 1, 08222, Barcelona, Spain.\\
\end{small}

\begin{abstract}
In this paper we study the development in Taylor series of the function $f(x)=x^x$. Section \textbf{\ref{sec01}} establishes a recursive relationship between successive derivatives of $f$ by using the coefficients $\olipt$ defined therein. From recursion between the derivatives you get one general description of them (section \textbf{\ref{sec02}}). Finally, section \textbf{\ref{sec03}} has the main result, the expansion series of $f$.\\
Section \textbf{\ref{sec04}} deals with numbers $\olipt$: characterization \textbf{(\ref{sec04.01})}, its relationship with rencontre numbers (\ref{sec04.02}) and its emergence as coefficients in certain polynomial families. The specific use of some of these polynomials allows eventually go deeper in the description of the series.
\end{abstract}
\section{Recursion of the derivatives of $f(x)=x^x$}\label{sec01}
For the series expansion of the function one needs a description of its derivatives, but $f'(x)$ is not immediately determined. Instead of deriving $f$, $f'$ can be found deriving $\ln(f(x))$.\\
On one hand,
$$\frac{\partial }{\partial x}\ln(f(x))=\frac{f'(x)}{f(x)}$$
on the other hand,
$$\ln(f(x))=\ln(x^x)=x\ln(x)$$
and therefore 
$$\frac{\partial }{\partial x}\ln(f(x))=\ln(x)+1$$
Joining both equalities one gets
\begin{equation}\label{eq01}
f'(x)=f(x)\ln(x)+f(x)
\end{equation}
Next derivatives can be found deriving successively $f'$. Note that $f$ appears in the expression of $f'$. This fact provides a recursive relationship between the derivatives of $f$, relationship characterized from the numbers $\olipt$ defined below:

\begin{definition}\label{oli01}
Given $a\in\mathbb{Z}_{>0}$ and $b\in\mathbb{Z}_{\geq 0}$, $\oli{a}{b}$ is defined by
$$\begin{cases}
\oli{a}{0}=1\\
\oli{a}{1}=\oli{a-1}{1}+1\\
\oli{a}{b}=\oli{a-1}{b}-(b-1)\oli{a-1}{b-1}\mbox{ if } b=2,...,a-1\\
\oli{a}{b}=0 \mbox{ if } b>a
\end{cases}$$ 
\end{definition}
Numbers $\olipt$ allow to relate the derivatives of $f$, as shown in the following proposition:
\begin{proposition}\label{oli02}
Let $f^{(n)}(x)$ be the $n-$th derivative of $f(x)=x^x$, then
\begin{equation}\label{eq02}
f^{(n)}(x)=f^{(n-1)}(x)\ln(x)+\sum_{i=0}^{n-1}\oli{n}{i}\frac{f^{(n-1-i)}(x)}{x^i}
\end{equation}
\end{proposition}
\begin{proof}
By induction. If $n=1$, then
$$f'(x)=f(x)\ln(x)+\oli{1}{0}\frac{f(x)}{x^0}=f(x)\ln(x)+f(x)$$
which is right (\ref{eq01}). Suppose now that the equation (\ref{eq02}) is right for $n$ and derive it:
$$f^{(n+1)}(x)=f^{(n)}(x)\ln(x)+\frac{f^{(n-1)}(x)}{x}$$
$$+\sum_{i=0}^{n-1}\oli{n}{i}\frac{f^{(n-i)}(x)x^i-f^{(n-1-i)}(x)ix^{i-1}}{x^{2i}}$$
The summation can be simplified separating it into two parts and resetting indexes in one of them:
$$\sum_{i=0}^{n-1}\oli{n}{i}\frac{f^{(n-i)}(x)x^i-f^{(n-1-i)}(x)ix^{i-1}}{x^{2i}}$$
$$=\sum_{i=0}^{n-1}\oli{n}{i}\frac{f^{(n-i)}(x)}{x^{i}}-
\sum_{i=0}^{n-1}\oli{n}{i}\frac{f^{(n-1-i)}(x)i}{x^{i+1}}$$
$$=\sum_{i=0}^{n-1}\oli{n}{i}\frac{f^{(n-i)}(x)}{x^{i}}-
\sum_{i=1}^{n}\oli{n}{i-1}\frac{f^{(n-i)}(x)(i-1)}{x^{i}}$$
$$=\oli{n}{0}f^{(n)}(x)+\sum_{i=1}^{n-1}\oli{n}{i}\frac{f^{(n-i)}(x)}{x^{i}}$$
$$-\oli{n}{n-1}\frac{f(x)(n-1)}{x^{n}}-\sum_{i=1}^{n-1}\oli{n}{i-1}\frac{f^{(n-i)}(x)(i-1)}{x^{i}}$$
$$=\oli{n}{0}f^{(n)}(x)-\oli{n}{n-1}\frac{f(x)(n-1)}{x^{n}}$$
$$+\sum_{i=1}^{n-1}\frac{f^{(n-i)}(x)}{x^{i}}\left(\oli{n}{i}-(i-1)\oli{n}{i-1}\right)$$
Now we can apply the definition of $\olipt$ to write the last summation in a more appropriate way:
$$\sum_{i=1}^{n-1}\frac{f^{(n-i)}(x)}{x^{i}}\left(\oli{n}{i}-(i-1)\oli{n}{i-1}\right)$$
$$=\frac{f^{(n-1)}(x)}{x}\oli{n}{1}+\sum_{i=2}^{n-1}\frac{f^{(n-i)}(x)}{x^{i}}\left(\oli{n}{i}-(i-1)\oli{n}{i-1}\right)$$
$$=\frac{f^{(n-1)}(x)}{x}\oli{n}{1}+\sum_{i=2}^{n-1}\frac{f^{(n-i)}(x)}{x^{i}}\oli{n+1}{i}$$
and replacing in the expression of $f^{(n+1)}(x)$, we have:
$$f^{(n+1)}(x)=f^{(n)}(x)\ln(x)+\frac{f^{(n-1)}(x)}{x}+\oli{n}{0}f^{(n)}(x)$$
$$-\oli{n}{n-1}\frac{f(x)(n-1)}{x^{n}}+\frac{f^{(n-1)}(x)}{x}\oli{n}{1}$$
$$+\sum_{i=2}^{n-1}\frac{f^{(n-i)}(x)}{x^{i}}\oli{n+1}{i}$$

$$f^{(n)}(x)\ln(x)+\frac{f^{(n-1)}(x)}{x}(\oli{n}{1}+1)+\oli{n}{0}f^{(n)}(x)$$
$$-\oli{n}{n-1}\frac{f(x)(n-1)}{x^{n}}+\sum_{i=2}^{n-1}\frac{f^{(n-i)}(x)}{x^{i}}\oli{n+1}{i}$$
Finally, by definition of $\olipt$, we know that $\oli{n}{1}+1=\oli{n+1}{1}$, $\oli{n}{0}=\oli{n+1}{0}$ and $-\oli{n}{n-1}(n-1)=\oli{n+1}{n}$ so
$$f^{(n+1)}(x)=f^{(n)}(x)\ln(x)+\frac{f^{(n-1)}(x)}{x}\oli{n+1}{1}+\oli{n+1}{0}f^{(n)}(x)$$
$$+\oli{n+1}{n}\frac{f(x)}{x^{n}}+\sum_{i=2}^{n-1}\frac{f^{(n-i)}(x)}{x^{i}}\oli{n+1}{i}$$
$$=f^{(n)}(x)\ln(x)+\sum_{i=0}^{n}\frac{f^{(n-i)}(x)}{x^{i}}\oli{n+1}{i}$$

\end{proof}

\section{General expression of the derivatives of $x^x$}\label{sec02}
In section \ref{sec01} we have obtained a recursive relationship of the derivatives of $x^x$. Through this recursion   and the next definition, a general direct (non-recursive) expression of the derivatives can be found (proposici\'on \ref{delta02}):
\begin{definition}\label{delta01}
For all $n,k\in \mathbb{Z}_{\geq 0}$ we define
$$\begin{cases}
\del{n}{0}=1\\
\del{n}{k}=\del{n-1}{k}+\sum_{i=0}^{k-1}\oli{n}{i}\frac{1}{x^i}\del{n-1-i}{k-1-i} \mbox{ if }k=1,\dots,n\\
\del{n}{k}=0\mbox{ if } k>n \end{cases}$$
\end{definition}
\begin{proposition}\label{delta02}
Let $f^{(n)}(x)$ be the $n-$th derivative of the function $f(x)=x^x$. Then
\begin{equation}\label{eq03}
f^{(n)}(x)=f(x)\sum_{i=0}^{n}\del{n}{i}\ln^{n-i}(x)
\end{equation}
\end{proposition}

\begin{proof} 
As in the previous proposition, by induction. If $n=1$, according to (\ref{eq03}) we get 
$$f'(x)=f(x)\left( \del{1}{0}\ln(x)+\del{1}{1}\right)$$
which is right due to $\del{1}{0}=\del{1}{1}=1$.\\
Let's suppose now that (\ref{eq03}) is right for $1,\dots,n$. From proposition \ref{oli02} we know
$$f^{(n+1)}(x)=f^{(n)}(x)\ln(x)+\sum_{i=0}^{n}\oli{n+1}{i}\frac{1}{x^i}f^{(n-i)}(x)$$
$$=\underbrace{f^{(n)}(x)\left[ \ln(x)+\oli{n+1}{0} \right] }_{A}+\underbrace{\sum_{i=1}^{n}\oli{n+1}{i}\frac{1}{x^i}f^{(n-i)}(x)}_{B}$$
On one hand, because of the induction hypothesis,
$$A=\left[ \ln(x)+\oli{n+1}{0} \right]f(x)\sum_{j=0}^{n}\del{n}{j}\ln^{n-j}(x)$$
$$=f(x)\left[\sum_{j=0}^{n}\del{n}{j}\ln^{n+1-j}(x)+\sum_{j=0}^{n}\del{n}{j}\ln^{n-j}(x)\oli{n+1}{0}\right]$$
$$=f(x)\left[\del{n}{0}\ln^{n+1}(x)+\del{n}{n}\oli{n+1}{0}\right.$$
$$+\left.\sum_{j=1}^{n}\del{n}{j}\ln^{n+1-j}(x)+\sum_{j=0}^{n-1}\del{n}{j}\ln^{n-j}(x)\oli{n+1}{0}\right]$$
$$=f(x)\left[\del{n}{0}\ln^{n+1}(x)+\del{n}{n}\oli{n+1}{0}\right.$$
$$+\left.\sum_{j=0}^{n-1}\del{n}{j+1}\ln^{n-j}(x)+\sum_{j=0}^{n-1}\del{n}{j}\ln^{n-j}(x)\oli{n+1}{0}\right]$$
$$=f(x)\left[\del{n}{0}\ln^{n+1}(x)+\del{n}{n}\oli{n+1}{0}\right.$$
$$+\left.\sum_{j=0}^{n-1}\left(\del{n}{j+1}+\oli{n+1}{0}\del{n}{j}\right)\ln^{n-j}(x)\right]$$
On the other hand,
$$B=\sum_{i=1}^{n}\oli{n+1}{i}\frac{1}{x^i}\sum_{j=0}^{n-i}\del{n-i}{j}f(x)\ln^{n-i-j}(x)$$
$$=\sum_{i=1}^{n}\sum_{j=0}^{n-i}\oli{n+1}{i}\frac{1}{x^i}\del{n-i}{j}f(x)\ln^{n-i-j}(x)$$
now, rearranging the summands of $B$ according to their degree $g$ in $\ln^g(x)$,
$$B=\sum_{g=0}^{n-1}\sum_{h=1}^{n-g}\oli{n+1}{h}\frac{1}{x^h}\del{n-h}{n-g-h}f(x)\ln^g(x)$$
With these arrangements, we can sum again $A$ and $B$:
$$A+B=\del{n}{0}f(x)\ln^{n+1}(x)+\oli{n+1}{0}\del{n}{n}f(x)$$
$$+\sum_{j=0}^{n-1}\left[ \del{n}{j+1}+\oli{n+1}{0}\del{n}{j} \right]f(x)\ln^{n-j}(x)$$
$$+\sum_{g=0}^{n-1}f(x)\ln^g(x)\sum_{h=1}^{n-g}\oli{n+1}{h}\frac{1}{x^h}\del{n-h}{n-g-h}$$
$$=\del{n}{0}f(x)\ln^{n+1}(x)+f(x)\sum_{h=0}^{n}\oli{n+1}{h}\frac{1}{x^h}\del{n-h}{n-h}$$
$$+\sum_{g=1}^{n-1}f(x)\ln^g(x)\sum_{h=1}^{n-g}\oli{n+1}{h}\frac{1}{x^h}\del{n-h}{n-g-h}$$
$$+\sum_{j=0}^{n-1}\left[ \del{n}{j+1}+\oli{n+1}{0}\del{n}{j} \right]f(x)\ln^{n-j}(x)$$
With the change $n-j=g$, $A+B$ can be writen as follows
$$A+B=\del{n}{0}f(x)\ln^{n+1}(x)+f(x)\sum_{h=0}^{n}\oli{n+1}{h}\frac{1}{x^h}\del{n-h}{n-h}$$
$$+\sum_{j=1}^{n-1}\left[ \del{n}{j+1}+\sum_{h=0}^{j}\oli{n+1}{h}\frac{1}{x^h}\del{n-h}{j-h} \right]f(x)\ln^{n-j}(x)$$
$$+\left[ \del{n}{1}+\oli{n+1}{0}\del{n}{0} \right]f(x)\ln^{n}(x)$$
which rearranged is
$$A+B=\del{n}{0}f(x)\ln^{n+1}(x)+
\left[ \del{n}{1}+\oli{n+1}{0}\del{n}{0} \right]f(x)\ln^{n}(x)+$$
$$\sum_{j=1}^{n-1}\left[ \del{n}{j+1}+\sum_{h=0}^{j}\oli{n+1}{h}\frac{1}{x^h}\del{n-h}{j-h} \right]f(x)\ln^{n-j}(x)+$$
$$f(x)\left[\sum_{h=0}^{n}\oli{n+1}{h}\frac{1}{x^h}\del{n-h}{n-h}\right]$$
Finally, due to
$$\del{n+1}{i}=\del{n}{i}+\sum_{h=0}^{i-1}\oli{n+1}{h}\frac{1}{x^h}\del{n-h}{i-1-h}$$
and 
$$\del{n+1}{0}=\del{n}{0}$$
the adjustment of the expression can be completed: 
$$A+B=\del{n+1}{0}f(x)\ln^{n+1}(x)+\del{n+1}{1}f(x)\ln^n(x)$$
$$+\sum_{j=1}^{n-1}\del{n+1}{j+1}\ln^{n-j}(x)+\del{n+1}{n+1}$$
$$=\del{n+1}{0}f(x)\ln^{n+1}(x)+\del{n+1}{1}f(x)\ln^n(x)$$
$$+\sum_{j=2}^{n}\del{n+1}{j}\ln^{n+1-j}(x)+\del{n+1}{n+1}$$
$$=\sum_{j=0}^{n+1}\del{n+1}{j}\ln^{n+1-j}(x)$$

\end{proof}

\section{Series expansion of $x^x$}\label{sec03}
In secci\'on \ref{sec02} the derivatives of $f(x)=x^x$ has been determined such a way they only depend on the series $\del{n}{k}$ and the function $f(x)$ itself. This characterization allows us to write the series expansion of $x^x$ in the same terms:
\begin{equation}\label{eq04}
f(x)=f(a)\sum_{k=0}^{\infty}\frac{(x-a)^k}{k!}\sum_{i=0}^{k}\dela{k}{i}\ln^{k-i}(a),\; a\in\mathbb{R}
\end{equation}
The particular case $a=1$ gives a very simplified expansion series. As $\ln^{k-i}(1)=0$ if $i<k$, every term except one are $0$. Besides, $f(1)=1$, so the result is
\begin{equation}\label{eq05}
f(x)=\sum_{k=0}^{\infty}\frac{(x-1)^k}{k!}\delx{k}{k}{1}
\end{equation}
In subsection \ref{sec04.03} we will get a more accurate expression of $\del{n}{k}$. This fact will allow us to rewrite the expansion series of $f$.
\begin{figure}
\label{serie}
$$\begin{array}{||c||c|c|c||}
\hline\hline
n\setminus x&0.5&0.9&2\\
\hline\hline
5&0.7057292&0.9095325&3.916667\\
10&0.7070978&0.9095326&4.005655\\
20&0.7071066&0.9095326&3.997326\\
\hline\hline
\end{array}$$
\caption{Few computations of the expansion series (\ref{eq05}) by sum from $k=0$ to $n$ for several values of $n$ and  $x$.}
\end{figure}

\section{Properties of numbers $\olipt$}\label{sec04}
According with the definition of $\olipt$ in section \ref{sec01}, several specific values of them can be computed easily, for instance
$$\oli{n}{1}=\oli{n-1}{1}+1=\oli{n-2}{1}+2=\cdots=\oli{1}{1}+n-1=n-1$$
but it would be interesting to get an expresi\'on allowing their computation directly. We will see below that studying sums of the numbers $\oli{a}{b}$ for any $b$ fixed, allows us to find this expression. It will be the main result in  subsecci\'on \ref{sec04.01}.\\
In \ref{sec04.02} we will see the relation between numbers $\olipt$ and rencontres numbers.  Finally, in subsection \ref{sec04.03} will be shown two examples of use for $\olipt$ . These examples are two families of polynomials. The first of them arises from a differential equation with the restriction of polynomials as solutions. The second family are polynomials that allow to write $\del{n}{k}$ in a more accurated way.  
\begin{figure}
\centering
\begin{tabular}{||c|ccccccccc||}
\hline\hline
$\oli{a}{b}$&$b=0$&$b=1$&$b=2$&$b=3$&$b=4$&$b=5$&$b=6$&$b=7$&$b=8$\\
\hline
$a=1$&1&0&0&0&0&0&0&0&0\\
$a=2$&1&1&0&0&0&0&0&0&0\\
$a=3$&1&2&-1&0&0&0&0&0&0\\
$a=4$&1&3&-3&2&0&0&0&0&0\\
$a=5$&1&4&-6&8&-6&0&0&0&0\\
$a=6$&1&5&-10&20&-30&24&0&0&0\\
$a=7$&1&6&-15&40&-90&144&-120&0&0\\
$a=8$&1&7&-21&70&-210&504&-840&720&0\\
$a=9$&1&8&-28&112&-420&1344&-3360&5760&-5040\\
\hline\hline
\end{tabular}
\caption{Values $\oli{a}{b}$ with $a=1,...,9$ and $b=0,...,8$}
\end{figure}

\subsection{Partial sums of numbers $\olipt$}\label{sec04.01}
\begin{definition} For any fixed $i\in\mathbb{Z}_{\geq 0}$, the partial sum of $n$ terms $\olipt$ is defined by
\begin{equation}
S_i(n)=\sum_{j=1}^{n}\oli{j}{i}=\sum_{j=i+1}^{n}\oli{j}{i}
\end{equation}
\end{definition}
$S_0(n)$ is easy to compute due to $\oli{a}{0}=1$ for all $a\in\mathbb{Z}_{>1}$, and so $S_0(n)=n$. $S_1(n)$ is also easy to compute:
$$S_1(n)=\sum_{j=2}^{n}\oli{j}{1}=\sum_{j=2}^{n}(j-1)=\sum_{j=1}^{n-1}j=\frac{(n-1)n}{2}$$
Again, these sums keep a recursive relationship:

\begin{lemma}\label{sumas01}
Partial sums $S_i(n)$ satisfy:
\begin{equation}\label{eq06}
S_i(n)=S_i(n-1)-(i-1)S_{i-1}(n-1)
\end{equation}
\end{lemma}
\begin{proof} 
By the definition of $\olipt$, $S_i(n)$ can be writen as follows
$$S_i(n)=\sum_{j=i+1}^{n}\oli{j}{i}$$
$$=\sum_{j=i+1}^{n}\left( \oli{j-1}{i}-(i-1)\oli{j-1}{i-1}\right)$$
$$=\sum_{j=i+1}^{n}\oli{j-1}{i}-(i-1)\sum_{j=i+1}^{n}\oli{j-1}{i-1}$$
$$=\sum_{j=i+2}^{n}\oli{j-1}{i}-(i-1)\sum_{j=i+1}^{n}\oli{j-1}{i-1}$$
$$=\sum_{j=i+1}^{n-1}\oli{j}{i}-(i-1)\sum_{j=i}^{n-1}\oli{j}{i-1}$$
concluding
$$S_i(n)=S_i(n-1)-(i-1)S_{i-1}(n-1)$$
when appliying again the definici\'on of $S_i(n)$.

\end{proof}

\begin{figure}\label{taula01}
\centering
\begin{tabular}{||c|c||}
\hline\hline
$i$&$S_i(n)$\\
\hline
&\\
$0$&$n$\\
&\\
$1$&$\frac{(n-1)n}{2}$\\
&\\
$2$&$-\frac{(n-2)(n-1)n}{6}$\\
&\\
$3$&$\frac{(n-3)(n-2)(n-1)n}{12}$\\
&\\
$4$&$-\frac{(n-4)(n-3)(n-2)(n-1)n}{20}$\\
&\\
\hline\hline
\end{tabular}
\caption{$S_i(n)$ for $i$ from $0$ to $4$}
\end{figure}
The relationship among partials sums leads to the next lemma:
\begin{lemma}\label{sum1}
$$S_i(n)=(-1)^{i+1}(i-1)!\binom{n}{i+1}\mbox{ if }i\in\mathbb{Z}_{>0}$$
\end{lemma}

\begin{proof} 
By induction. Let's suppose 
$$S_i(n)=(-1)^{i+1}(i-1)!\binom{n}{i+1}\mbox{ if }i\in\mathbb{Z}_{>0}$$ 
(clearly right for $i=1,...,4$ as can be seen in figure \ref{taula01}). Then
$$S_{i+1}(n)=S_{i+1}(n-1)-iS_i(n-1)$$
$$=S_{i+1}(n-2)-iS_i(n-2)-iS_i(n-1)$$
$$=\cdots=S_{i+1}(i+1)-i\sum_{j=i+1}^{n-1}S_i(j)$$
and as $S_{i+1}(i+1)=0$,
$$S_{i+1}(n)=-i\sum_{j=i+1}^{n-1}S_i(j)$$
$$=-i\sum_{j=i+1}^{n-1}(-1)^{i+1}(i-1)!\binom{j}{i+1}$$
$$=\sum_{j=i+1}^{n-1}(-1)^{i+2}i!\binom{j}{i+1}$$
$$=(-1)^{i+2}i!\sum_{j=i+1}^{n-1}\binom{j}{i+1}$$
but
\begin{equation}\label{numcomb}
\sum_{j=i+1}^{n-1}\binom{j}{i+1}=\binom{n}{i+2}
\end{equation}
and therefore 
$$S_{i+1}(n)=(-1)^{i+2}i!\binom{n}{i+2}$$
To finish this proof, previous equation \ref{numcomb} must be proven:\\
If $n=i+2$ the equality is right. Let's suppose now that the equality is right from $i+2$ to $n-2$.
$$\sum_{j=i+1}^{n-1}\binom{j}{i+1}=\sum_{j=i+1}^{n-2}\binom{j}{i+1}+\binom{n-1}{i+1}$$
and by induction hypothesis,
$$\sum_{j=i+1}^{n-1}\binom{j}{i+1}=\binom{n-1}{i+2}+\binom{n-1}{i+1}$$
Finally, by Pascal's identity \cite{Kos} we obtain the result.

\end{proof}

As a consequence of lemma \ref{sum1} we can get also a more accurate expression of $\oli{a}{b}$:
\begin{proposition}\label{num1} For any $a,b\in\mathbb{Z}_{>0}$ with $a>b$, 
\begin{equation}\label{eq07}
\oli{a}{b}=(-1)^{b+1}(b-1)!\binom{a-1}{b}\mbox{ if }b\in\mathbb{Z}_{>0}
\end{equation}
\end{proposition}

\begin{proof} 
If $b=1$,
$$(-1)^2 1!\binom{a-1}{1}=a-1=\oli{a}{1}$$
If $b>1$,
$$\oli{a}{b}=\oli{a-1}{b}-(b-1)\oli{a-1}{b-1}$$
$$=\oli{a-2}{b}-(b-1)\oli{a-2}{b-1}-(b-1)\oli{a-1}{b-1}$$
$$=\cdots=\oli{b}{b}-(b-1)\oli{b}{b-1}-\cdots-(b-1)\oli{a-1}{b-1}$$
and as $\oli{b}{b}=0$, 
$$\oli{a}{b}=-(b-1)\sum_{j=b}^{a-1}\oli{j}{b-1}=-(b-1)S_{b-1}(a-1)$$
Finally, using lemma \ref{sum1},
$$\oli{a}{b}=-(b-1)(-1)^b (b-1)!\binom{a-1}{b}$$
$$=(-1)^{b+1}(b-1)!\binom{a-1}{b}$$

\end{proof}

Formula \ref{eq07} gives a recursive way for the computation of $\oli{a}{0}=1$, but in a different way than in definition \ref{oli01}:
\begin{corollary}\label{cor1}
\begin{equation}
\oli{a}{b}=-(b-1)\sum_{j=b}^{a-1}\oli{j}{b-1}\mbox{ if }b\in\mathbb{Z}_{>0}
\end{equation}
\end{corollary}

\begin{proof} 
Just using lemma \ref{sum1} and proposition \ref{num1}:
$$\oli{a}{b}=(-1)^{b+1}(b-1)!\binom{a-1}{b}$$
$$=-(b-1)(-1)^b (b-2)!\binom{a-1}{b}$$
$$=-(b-1)S_{b-1}(a-1)$$
$$=-(b-1)\sum_{j=b}^{a-1}\oli{j}{b-1}$$

\end{proof}

\subsection{Mapping between the numbers $\olipt$ and rencontres numbers}\label{sec04.02}
Rencontres numbers count how many permutations in a set $\{1,2,\dots,n\}$ has certain quantity of fixed points, i.e., how many derangements (partial o total). Specifically $D_{n,k}$ is the quantity of permutations of $n$ items with $k$ of them fixed. Rencontres numbers can be computed by
$$\begin{cases} D_{n,k}=\binom{a}{b}D_{n-k}\\D_{n}=n!\sum_{i=0}^{n}\frac{(-1)^i}{i!}\end{cases}$$
(see \cite{Rio}), which allows you to check that numbers $\olipt$ match with rencontres numbers except for a multiplicative factor.
\begin{proposition}\label{ren1}
For $n$ and $k$ non-negative integer numbers given with $n>k$, 
$$\oli{n}{k}=\lambda_k D_{n-1,n-k-1}$$
where $\lambda_k=\left(k\sum_{i=0}^{k}\frac{(-1)^{i+k+1}}{i!}\right)^{-1}$.\\
Equivalently,
$$D_{n,k}=(n-k)\sum_{i=0}^{n-k}\frac{(-1)^{i+n-k+1}}{i!}\oli{n+1}{n-k}$$
Note: $\lambda_k$ only depends on $k$ (as the notation indicates).
\end{proposition}

\begin{proof} 
Just propose equality
$$\oli{n}{k}=\lambda D_{n-1,n-k-1}$$ 
and find $\lambda$:
$$\lambda=\frac{ \oli{n}{k}}{D_{n-1,n-k-1}}$$
$$=\frac{(-1)^{k+1}(k-1)!\binom{n-1}{k}}{\binom{n-1}{n-k-1}k!\sum_{i=0}^{k}\frac{(-1)^i}{i!}}$$
$$=\frac{(-1)^{k+1}}{k\sum_{i=0}^{k}\frac{(-1)^i}{i!}}$$
$$=\left(k\sum_{i=0}^{k}\frac{(-1)^{i+k+1}}{i!}\right)^{-1}$$
and as $\lambda$ only depends on $k$, we can conclude $\lambda_k=\lambda$.

\end{proof}

\subsection{Polynomials with coefficients $\olipt$}\label{sec04.03}
In this section we will see two families of polynomials with numbers $\oli{a}{b}$ as coefficients. In the first example consider the family $P_n(x),\;n\geq 1$ of monic polynomials with degree $n-1$ satisfying the equation 
$$\frac{\partial P_n}{\partial x}(x)=(n-1)P_{n-1}(x)$$
The proposition below determines how these polynomials $P_n(x)$ are:
\begin{proposition} \label{pol01}
$$P_n(x)=\sum_{i=0}^{n-1}\oli{n}{i}x^{n-1-i}\mbox{ }\forall n\in\mathbb{N}$$
\end{proposition}
\begin{proof} 
Let $P_n(x)=\sum_{i=0}^{n-1}a_{n,i}x^{n-1-i}$ with $a_{n,i}\in \mathbb{R}$ and $a_{n,0}=1$ be the solution of the equation. By definition $P_1(x)=1=\oli{1}{0}$. Suppose that the proposition is right up to $n-1$.
Deriving $P_n$ one gets
$$\frac{\partial P_n(x)}{\partial x}=\sum_{i=0}^{n-1}(n-1-i)a_{n,i}x^{n-2-i}$$
$$=\sum_{i=0}^{n-2}(n-1-i)a_{n,i}x^{n-2-i}$$
but by definition of $P_n$ and induction hypothesis we have
$$\frac{\partial P_n}{\partial x}(x)=(n-1)P_{n-1}(x)$$
$$=(n-1)\sum_{i=0}^{n-2}\oli{n-1}{i}x^{n-2-i}$$
Equalizing both expressions we get the equations 
$$(n-1-i)a_{n,i}=(n-1)\oli{n-1}{i}\;,i=0,\dots,n-2$$
By proposition \ref{num1}, these equations can be writen
$$a_{n,i}=\frac{n-1}{n-1-i}(-1)^{i+1}(i-1)!\binom{n-2}{i}$$
$$=\frac{(n-1)(n-2)!}{(n-1-i)i!(n-2-i)!}(-1)^{i+1}(i-1)!\binom{n-2}{i}$$
$$=\frac{(n-1)!}{i!(n-1-i)!}(-1)^{i+1}(i-1)!\binom{n-2}{i}$$
$$=(-1)^{i+1}(i-1)!\binom{n-1}{i}=\oli{n}{i}$$

\end{proof}

More in general we can write the relationship between polynomials $P_n(x)$ and their derivatives:
\begin{proposition}\label{pol02}
for $0<r\leq k$ given,
$$\frac{\partial^k}{\partial x^k}P_n(x)=(-1)^{r+1}r\oli{n}{r}\frac{\partial^{k-r}}{\partial x^{k-r}}P_{n-r}(x)$$
\end{proposition} 

\begin{proof} 
Just applying proposition \ref{pol01} as much as needed:
$$\frac{\partial^k}{\partial x^k}P_n(x)=(n-1)\frac{\partial^{k-1}}{\partial x^{k-1}}P_{n-1}(x)$$
$$=(n-1)(n-2)\frac{\partial^{k-2}}{\partial x^{k-2}}P_{n-2}(x)$$
$$=\cdots=(n-1)(n-2)\cdots (n-r)\frac{\partial^{k-r}}{\partial x^{k-r}}P_{n-r}(x)$$
We can write it in factorial terms: 
$$\frac{\partial^k}{\partial x^k}P_n(x)=\frac{(n-1)!}{(n-r-1)!}\frac{\partial^{k-r}}{\partial x^{k-r}}P_{n-r}(x)$$
which can be expressed by combinatorial numbers and use (\ref{eq07}) to get the desired result:
$$\frac{\partial^k}{\partial x^k}P_n(x)=\binom{n-1}{r}r!\frac{\partial^{k-r}}{\partial x^{k-r}}P_{n-r}(x)$$
$$=(-1)^{r+1}r\oli{n}{r}\frac{\partial^{k-r}}{\partial x^{k-r}}P_{n-r}(x)$$

\end{proof}

Next corollaries of proposition \ref{pol02} show specific cases of it, the specific cases $r=k$ and $r=k=n-1$. 
\begin{corollary}\label{pol02_01} For all $k\geq 0$,
$$\frac{\partial^{k}}{\partial x^{k}}P_n(x)=(-1)^{k+1} k\oli{n}{k}P_{n-k}(x)$$
\end{corollary}

\begin{proof} 
As said previously, it's just proposition \ref{pol02} when $k=r$.

\end{proof}

\begin{corollary}\label{pol02_02} For all $n\geq 1$,
$$\frac{\partial^{n-1}}{\partial x^{n-1}}P_n(x)=\Gamma(n-1)=(n-1)!$$
\end{corollary}

\begin{proof} 
Just use proposition \ref{pol02} in the specific case $k=r=n-1$ and consider $P_1(x)=1$.

\end{proof}

As second example there are polynomials definable by numbers $\olipt$ which are useful in the study of $\del{n}{k}$ series:
\begin{definition}\label{polQ}
\begin{equation}\label{polQ01}
Q_k(x)=x\sum_{i=0}^{k-1}\oli{k}{i}Q_{k-1-i}(x)\mbox{ for }k=1,2,\dots
\end{equation}
where $Q_0(x)=\frac{1}{x}$
\end{definition}
Once defined polynomials $Q_k(x)$, we can find a more accurate expression of $\del{n}{k}$:
\begin{proposition}\label{pdel01}
\begin{equation}\label{delta03}
\del{n}{k}=\binom{n}{k}\frac{Q_k(x)}{x^{k-1}}
\end{equation}
\end{proposition}

\begin{proof}
For the proof of the proposition we need a lemma:
\begin{lemma}\label{lema01}
$$\oli{n}{i}\binom{n-1-i}{k-i}=\oli{k+1}{i}\binom{n-1}{k}$$
\end{lemma}
\begin{proof}(of the lemma \ref{lema01})\\ \\
Just applying proposition \ref{num1} and operate:
$$\oli{n}{i}\binom{n-1-i}{k-i}=(-1)^{i+1}\binom{n-1}{i}(i-1)!\binom{n-1-i}{k-i}$$
$$=(-1)^{i+1}\frac{(n-1)!}{i!(n-1-i)!}(i-1)!\frac{(n-1-i)!}{(k-i)!(n-k-1)!}$$
$$=(-1)^{i+1}\frac{(n-1)!(i-1)!}{i!(k-i)!(n-k-1)!}\frac{k!}{k!}$$
$$=(-1)^{i+1}\binom{k}{i}\binom{n-1}{k}(i-1)!$$
$$=\oli{k+1}{i}\binom{n-1}{k}$$

\end{proof}
According to the proposition, 
$$\del{n}{1}=\binom{n}{1}\frac{Q_1(x)}{x^{1-1}}=nQ_1(x)$$
but by definition 
$$Q_1(x)=x\oli{1}{0}Q_0(x)=x\frac{1}{x}=1$$ 
Therefore $\del{1}{1}=n\cdot 1=n$, which is true.\\
For $n$ and $k$ fixed, suppose that $\del{m}{r}$ satisfies the proposition for every $r$ if $m<n$ and $\del{n}{r}$ also for all $r\leq k$. Then,
$$\del{n}{k+1}=\del{n-1}{k+1}+\sum_{i=0}^{k}\oli{n}{i}\frac{1}{x^i}\del{n-1-i}{k-i}$$
and applying the hypothesis we get the equation,
$$\del{n}{k+1}=\binom{n-1}{k+1}\frac{Q_{k+1}(x)}{x^{k}}+\sum_{i=0}^{k}\oli{n}{i}\frac{1}{x^i}\binom{n-1-i}{k-i}\frac{Q_{k-i}(x)}{x^{k-i-1}}$$
$$=\binom{n-1}{k+1}\frac{Q_{k+1}(x)}{x^{k}}+\sum_{i=0}^{k}\oli{n}{i}\binom{n-1-i}{k-i}\frac{Q_{k-i}(x)}{x^{k-1}}$$
but by lemma \ref{lema01} the expression can be writen
$$\binom{n-1}{k+1}\frac{Q_{k+1}(x)}{x^{k}}+\sum_{i=0}^{k}\oli{k+1}{i}\binom{n-1}{k}\frac{xQ_{k-i}(x)}{x^{k}}$$
$$=\binom{n-1}{k+1}\frac{Q_{k+1}(x)}{x^{k}}+\binom{n-1}{k}\frac{x}{x^{k}}\sum_{i=0}^{k}\oli{k+1}{i}Q_{k-i}(x)$$
Now we can use the definition \ref{polQ} to compute the sum. The equation becomes:
$$\binom{n-1}{k+1}\frac{Q_{k+1}(x)}{x^{k}}+\binom{n-1}{k}\frac{Q_{k+1}(x)}{x^{k}}=
\left[\binom{n-1}{k+1}+\binom{n-1}{k}\right]\frac{Q_{k+1}(x)}{x^{k}}$$
Finally, by Pascal's identity \cite{Kos},  
$$\binom{n-1}{k+1}+\binom{n-1}{k}=\binom{n}{k+1}$$
which leads to
$$\del{n}{k+1}=\binom{n}{k+1}\frac{Q_{k+1}(x)}{x^{k}}$$

\end{proof}

As a consequence, the expansion series (\ref{eq04}) and (\ref{eq05}) in section \ref{sec03} can be writen as follows:

\begin{equation}
f(x)=f(a)\sum_{k=0}^{\infty}\frac{(x-a)^k}{k!}\sum_{i=0}^{k}\binom{k}{i}\frac{Q_k(a)}{a^{k-1}}\ln^{k-i}(a)
\end{equation}
for any $a\in \mathbb{R}$. In the specific case $a=1$,
\begin{equation}
f(x)=\sum_{k=0}^{\infty}\frac{(x-1)^k}{k!}Q_k(1)
\end{equation}

\end{document}